\documentclass{amsart}
\usepackage{amsfonts}
\usepackage[active]{srcltx}
\usepackage{amsmath}
\usepackage{amssymb}
\usepackage{amsthm}
\usepackage{mathtools}
\usepackage{shuffle}
\usepackage{bm} 

\usepackage{enumitem}
\usepackage{wasysym}
\usepackage{mathscinet}
\usepackage{verbatim}
\usepackage{graphicx}
\usepackage[
bookmarks=true,         
bookmarksnumbered=true, 
colorlinks=true, pdfstartview=FitV, linkcolor=blue, citecolor=blue,
urlcolor=blue]{hyperref}
\usepackage{microtype}
\theoremstyle{plain}
\newtheorem{theorem}{Theorem}[section]

\newtheorem{lemma}[theorem]{Lemma}
\newtheorem{problem}[theorem]{Problem}
\newtheorem{proposition}[theorem]{Proposition}

\renewenvironment{proof}[1][Proof]{\textbf{#1.} }{\ \rule{0.5em}{0.5em} \par }
\theoremstyle{remark}

\theoremstyle{definition}
\newtheorem{remark}[theorem]{Remark}
\newtheorem{definition}[theorem]{Definition}

\def\RR{\mathbb{R}}

\def\be{{\beta}}

\def\la{{\lambda}}

\def\De{{\Delta}}

\def\al{{\alpha}}

\def\be{{\beta}}

\def\ga{{\gamma}}

\def\De{{\Delta}}

\def\la{{\lambda}}

\def \eref#1{\hbox{(\ref{#1})}}

\def\al{{\alpha}}

\newcommand{\JJ}{\mathcal J}

\newcommand{\ep}{\ensuremath{\varepsilon}}

\renewcommand{\d}{d}

\def\dt{d t}

\def\dr{d r}
\def\ds{d s}

\let\Section=\section
\def\section{\setcounter{equation}{0}\Section}

\title{Nonlinear Young  integrals   via fractional calculus}
\date{January  2015}
\author[Y. Hu]{Yaozhong Hu}
\thanks{Y. Hu is partially supported by a grant from the Simons Foundation \#209206
and by a General Research Fund of University of Kansas.}
\address{Department of Mathematics \\
The University of Kansas \\
Lawrence, Kansas, 66045}
\email{yhu@ku.edu}
\author[K. Le]{Khoa Le}
\email{khoale@ku.edu}
\begin{document}
\begin{abstract} For H\"older continuous functions $W(t,x)$ and
$\varphi_t$, we define nonlinear integral $\int_a^b W(dt, \varphi_t)$ via fractional
calculus.   This nonlinear integral arises naturally in the Feynman-Kac formula for
 stochastic heat  equations with random coefficients \cite{hule2015}. We also
 define iterated nonlinear integrals.
\end{abstract}
\maketitle

\section{Introduction}
Let $\{\varphi_t\,, t\ge 0\}$ be a  H\"older continuous function
and let $\{W(t,x), t\ge 0\,, x\in \RR^d\}$ be another  jointly H\"older continuous function
of several variables.
In authors' recent paper \cite{hule2015} the nonlinear Young integral  $\int_a^b W(dt, \varphi_t)$
is introduced to establish the Feynman-Kac formula for general
stochastic partial differential equations with random coefficients,
namely, $
\partial _t u(t,x)+Lu(t,x) +u(t,x)\partial_tW(t,x)=0\,,$
where
$
Lu(t,x)=\frac12 \sum_{i,j=1}^d a_{ij}(t,x, W) \partial _{x_ix_j}^2
u(t,x) +\sum_{i =1}^d b_{i }(t,x, W) \partial _{x_i } u(t,x)
$
with the coefficients $a_{ij}$ and $b_i$ depending on $W$. In that paper, the
nonlinear Young integral  $\int_a^b W(dt, \varphi_t)$ was defined by approximation,
in particular, by the sewing lemma of \cite{FeyelPradelle06}.  In this paper, we
study the nonlinear Young integral  $\int_a^b W(dt, \varphi_t)$
by means of fractional calculus. This approach may provide more detailed properties
of the solutions to the  equations (see \cite{hunualart2007} and \cite{hunualart2009}).

To expand the solution of a (nonlinear) differential equation
with explicit remainder term we need to define (iterated) multiple
integrals (see \cite{hustochastics}). We shall also give a definition of the iterated nonlinear
Young integrals. Some elementary estimate is  also obtained.

The paper is organized as follows.
Section 2 briefly recall some preliminary material on fractional
calculus that are needed lately. Section 3 deals with nonlinear
Young integral and Section 4 is concerned with iterated
nonlinear  Young integral.

%

\section{Fractional integrals and derivatives}
In this section we recall some results from fractional calculus.

Let   $-\infty<a<b<\infty$,  $\al>0$ and $p\ge 1$ be  real numbers.
Denote by $L^p(a, b)$ the space of all measurable functions on $(a,
b)$ such that
\[
\|f\|_p  :=\left(\int_a^b |f(t)|^p dt\right)^{1/p} <\infty\,.
\]
Denote by  $C([a, b])$ the space of continuous functions on $[a, b]$. Let
$f\in L^{1}\left( [a,b]\right) $.   The left-sided  fractional
Riemann-Liouville integral  $I_{a+}^{\alpha }f $  is defined as
\begin{equation}
I_{a+}^{\alpha }f\left( t\right) =\frac{1}{\Gamma \left( \alpha \right) }%
\int_{a}^{t}\left( t-s\right) ^{\alpha -1}f\left( s\right) \ds\,,
\quad t\in (a, b)\label{e.2.1}
\end{equation}
and the    right-sided fractional Riemann-Liouville integral
$I_{b-}^{\alpha }f $ is defined as
\begin{equation}
I_{b-}^{\alpha }f\left( t\right) =\frac{\left( -1\right) ^{-\alpha
}}{\Gamma \left( \alpha \right) }\int_{t}^{b}\left( s-t\right) ^{\alpha -1}f\left( s\right) \ds\,,
\quad t\in (a, b)
\label{e.2.2}
\end{equation}
 where $\left( -1\right) ^{-\alpha }=e^{-i\pi \alpha }$ and $%
\Gamma \left( \alpha \right) =\int_{0}^{\infty }r^{\alpha -1}e^{-r}\dr$ is
the Euler gamma function. Let $I_{a+}^{\alpha }(L^{p})$ (resp. $%
I_{b-}^{\alpha }(L^{p})$) be the image of $L^{p}(a,b)$ by the operator $%
I_{a+}^{\alpha }$ (resp. $I_{b-}^{\alpha }$). If $f\in I_{a+}^{\alpha
}\left( L^{p}\right) \ $ (resp. $f\in I_{b-}^{\alpha }\left( L^{p}\right) $)
and $0<\alpha <1$,  then the (left-sided or right-sided) Weyl derivatives are defined
(respectively) as
\begin{equation}
D_{a+}^{\alpha }f\left( t\right) =\frac{1}{\Gamma \left( 1-\alpha \right) }%
\left( \frac{f\left( t\right) }{\left( t-a\right) ^{\alpha }}+\alpha
\int_{a}^{t}\frac{f\left( t\right) -f\left( s\right) }{\left(
t-s\right) ^{\alpha +1}}\ds\right)  \label{e.2.3}
\end{equation}
and
\begin{equation}
D_{b-}^{\alpha }f\left( t\right) =\frac{\left( -1\right) ^{\alpha }}{\Gamma
\left( 1-\alpha \right) }\left( \frac{f\left( t\right) }{\left( b-t\right)
^{\alpha }}+\alpha \int_{t}^{b}\frac{f\left( t\right) -f\left( s\right) }{%
\left( s-t\right) ^{\alpha +1}}\ds\right)\,,  \label{e.2.4}
\end{equation}%
where $a\leq t\leq b$ (the convergence of the integrals at the singularity $%
s=t$ holds point-wise for almost all $t\in \left( a,b\right) $ if $p=1$ and
moreover in $L^{p}$-sense if $1<p<\infty $).

It is clear that if $f$ is H\"older continuous of order $\mu>\al$, then
the two Weyl derivatives exist.

For any $\be \in (0,1)$, we denote by $C^{\be }([a,b])$ the space of $%
\be $-H\"{o}lder continuous functions on the interval $[a,b]$. We will
make use of the notation %
\[
\| f\| _{\beta;a,b }=\sup_{a< \theta <r< b}\frac{%
|f(r)-f(\theta )|}{|r-\theta |^{\beta }}
\]
 (which is a seminorm)  and
\[
\|f\|_{\infty;a,b }=\sup_{a\leq r\leq b}|f(r)|,
\]
where $f:\mathbb{R} \rightarrow \mathbb{R}$ is a given continuous
function.

It is well-known that $C^{\be }([a,b])$ with the H\"older  norm
$
\| f\| _{\beta;a,b }  +\| f\|_{\infty;a,b }$ is  a complete  Banach space.
 But it is {\it not} separable

 Using the fractional calculus, we have
 \begin{proposition}
 \label{p1} Let $0<\al<1$.    If $f$ and $g$ are continuously differentiable functions on the interval $[a, b]$, then
 \begin{equation}
 \int_{a}^{b}f\d g=(-1)^{\alpha }\int_{a}^{b}\left( D_{a+}^{\alpha }f\left( t\right)\right)
 \left( D_{b-}^{1-\alpha }g_{b-}\left( t\right) \right) \dt,  \label{e.2.14}
 \end{equation}%
 where $g_{b-}\left( t\right) =g\left( t\right) -g\left( b\right) $.
 \end{proposition}
In what follows $\kappa$ denotes a universal generic
constant depending only on $\la, \tau, \al$
and independent of $W$, $\varphi$ and $a, b$. The value
of $\kappa$ may vary from occurrence to occurrence

For two function $f, g:[a, b]\rightarrow \RR$, we can define the Riemann-Stieltjes
integral $\int_a^b f(t) \d g(t)$.  Here we  recall a  result which is well-known
(see for example \cite{zahle},    \cite{hustochastics}
or \cite{hunualart2007}, \cite{hunualart2009}).
\begin{lemma}\label{lemma.fdg}
Let $f$ and $g$ be H\"older continuous functions of orders $\alpha$
and $\beta$ respectively. Suppose that $\alpha+\beta>1$. Then the
Riemann-Stieltjes integral  $ \int_a^bf(t)\d g(t)$ exists and  for any $\ga\in(1- \beta,\alpha)$,  we have
\begin{equation}
\int_a^bf(t)\d g(t)=(-1)^\gamma \int_a^b D_a^{\ga } f(t) D_{b-}^{1-\ga }g_{b-} (t) \dt \,. \label{eq.young}
\end{equation}
Moreover,   there is a constant $\kappa$ such that
 \begin{equation}\label{est.fdg}
\left|\int_a^bf(t)\d g(t)   \right| \le \kappa
\|g\|_{\beta;a,b}(\|f\|_{ \infty; a,b}|b-a|^{\beta}+\|f\|_{ \alpha; a,b}|b-a|^{\alpha+\beta}).
\end{equation}
\end{lemma}
\begin{proof} We refer to \cite{zahle} or  \cite{hustochastics}  for a proof of \eqref{eq.young}. We
shall outline a    proof  of \eref{est.fdg}.
Let $\gamma$ be such that $\alpha>\gamma>1-\beta$. Applying
fractional integration by parts formula \eqref{eq.young}, we obtain
\begin{equation*}
\left|\int_a^bf(t)\d g(t)\right|\le
\int_a^b\!|D^\gamma_{a+}f(t)D^{1-\gamma}_{b-}g_{b-}(t)|\d t.
\end{equation*}
From \eqref{e.2.3} and \eqref{e.2.4} it is easy to see
\begin{equation*}
|D^{1-\gamma}_{b-}g_{b-}(t)|\le
\kappa\|g\|_{\beta;a,b}(b-r)^{\beta+\gamma-1}
\end{equation*}
and
\begin{equation*}
|D^\gamma_{a+}f(t)|\le \kappa[\|f\|_{\infty;a,b}(t-a)^{-\gamma}
+\|f\|_{\alpha;a,b}(t-a)^{\alpha-\gamma}].
\end{equation*}
Therefore
\begin{multline*}
\left|\int_a^bf(t)\d g(t)\right|\le
\kappa\|g\|_{\beta;a,b}\left(\|f\|_{\infty;a,b}\int_a^b\!(t-a)^{-\gamma}(b-t)^{\beta+\gamma-1}\d
t\right.\\\left.+\|f\|_{\alpha;a,b}\int_a^b\!(t-a)^{\alpha-\gamma}(b-t)^{\beta+\gamma-1}\d
t\right).
\end{multline*}
The integrals on the right hand side can be computed by making the
substitution $t=b-(b-a)s$, hence we derive \eqref{est.fdg}
\end{proof}

We also need the following lemma in the proof of our main results.
\begin{lemma}
\label{lem:ifst}
Let $ f(s,t), a\le  s< t\le b$ be a measurable
function of $s$ and $t$ such that
\[
\int_{a}^{b}\int_{a}^{t}\frac{|f(s,t)|}{(t-s)^{1-\alpha}}\ds\dt<\infty.
\]
Then
\[
\int_{a}^{b}I_{a+}^{\alpha,t}f(t,t')|_{t'=t}\dt=(-1)^{\alpha}\int_{a}^{b}I_{b-}^{\alpha,t'}f(t,t')|_{t'=t}\dt.
\]
\end{lemma}
\begin{proof}
An  application of Fubini's theorem yields
\begin{align*}
\int_{a}^{b}I_{a+}^{\alpha,t}f(t,t')|_{t'=t}\dt  &=  \frac{1}{\Gamma(\alpha)}\int_{a}^{b}\int_{a}^{t}\frac{f(s,t)}{(t-s)^{1-\alpha}}\ds\dt\\
  &= \frac{1}{\Gamma(\alpha)}\int_{a}^{b}\int_{s}^{b}\frac{f(s,t)}{(t-s)^{1-\alpha}}\dt\ds\\
  &=  (-1)^{\alpha}\int_{a}^{b}I_{b-}^{\alpha,t'}f(t,t')|_{t'=t}\dt
 \end{align*}
 which is the lemma.
\end{proof}

\setcounter{equation}{0}
\section{Nonlinear integral}\label{sec.pathint}
\global\long\def\ctau{C^{\tau}\left([t_{0}-T,t_{0}+T]\right)}

In this section we shall use fractional calculus to define
the (pathwise) nonlinear   integral $\int_{a}^{b}W(\dt,\varphi_{t})$.
This method  only relies
on regularity of the sample paths of $W$ and $\varphi$. More precisely, it is applicable to stochastic
processes with H\"older continuous sample paths.


Another advantage of this approach is that in the theory of stochastic processes
it is usually difficult to obtain almost sure  type of results.  If
the  sample paths of the process is H\"older continuous, then one can apply
this approach to each sample path and almost surely results are then automatic.

In what follows, we shall use $W$ to denote a deterministic function
 $W: \RR\times \RR^d\rightarrow \RR^d$.
We make the following assumption on the regularity of $W$
\begin{enumerate}[label=\bm{$(W)$}]
\item\label{cond.w} There are constants $\tau\,, \la\in (0, 1]$, $\beta\ge0$ such that
for all $a<b$, the seminorm
\begin{equation}\label{eq:cond.w2}
\begin{split}
  &\|W\|_{ \tau, \la; a, b }\\
  :&=\sup_{\substack{a\le s<t\le b\\ x,y\in \RR^d;x\neq y}}\frac{
\left|W(s,x)-W(t,x)-W(s,y) + W(t,y)\right|}{
  |t-s|^{\tau}|x-y|^{\lambda}}\\
 &\quad+ \sup_{\substack{a\le s<t\le b\\ x\in \RR^d}}\frac{
\left|W(s,x)-W(t,x)\right|}{
   |t-s|^{\tau}}+\sup_{\substack{a\le t\le b\\ x,y\in \RR^d;x\neq y}}\frac{
\left|W(t,y) - W(t,x)\right|}{
   |x-y|^{\lambda}} \,,
\end{split}
\end{equation}
is finite.
\end{enumerate}
About the function $\varphi$, we assume
\begin{enumerate}[label=$\bm{(\phi)}$]
  \item\label{cond.phi} $\varphi$ is locally H\"older continuous of order $\ga\in (0, 1]$. That is the seminorm
  \[
  \|\varphi\|_{\ga; a, b}=\sup_{a\le s<t\le b}\frac{|\varphi(t)-\varphi(s)|}{|t-s|^\ga}\,,
  \]
  is finite for every $a<b$.
\end{enumerate}
Among three terms appear in \ref{cond.w}, we will pay special attention to the first term. Thus, we denote
\begin{equation*}
  [W]_{\tau,\lambda;a,b}=\sup_{\substack{a\le s<t\le b\\ x,y\in \RR^d;x\neq y}}\frac{
\left|W(s,x)-W(t,x)-W(s,y) + W(t,y)\right|}{
   |t-s|^{\tau}|x-y|^{\lambda}}\,.
\end{equation*}
If $a, b$ is clear in the context, we frequently omit the dependence on $a, b$.
For instance,   $\|W\|_{  \tau, \la}$ is an abbreviation for $\|W\|_{  \tau, \la; a,b}$, $\|\varphi\|_{\gamma}$ is an abbreviation for $\|\varphi\|_{\gamma;a,b}$  and so on.
We shall assume that $a$ and $b$ are finite.  Thus it is easy
to see that for any $c\in [a, b]$
\[
\sup_{a\le t\le b}|\varphi(t)|=
\sup_{a\le t\le b}| \varphi(c)+\varphi(t)-\varphi(c)|
\le |\varphi(c)|+\| \varphi\|_{\ga} |b-a|^\ga<\infty\,.
\]
Thus assumption \ref{cond.phi} also implies that
\[
\|\varphi\|_{\infty; a, b}
:=\sup_{a\le t\le b}|\varphi(t)|<\infty\,.
\]
For the results presented in this section, the condition \ref{cond.w} can be relaxed to
\begin{enumerate}[label=\bm{$(W')$}]
\item \label{cond.wprime} There are constants $\tau\,, \la\in (0, 1]$, such that
for all $a<b$ and compact set $K$ in $\RR^d$, the seminorm
\begin{equation*}
\begin{split}
  &\sup_{\substack{a\le s<t\le b\\ x,y\in K;x\neq y}}\frac{
\left|W(s,x)-W(t,x)-W(s,y) + W(t,y)\right|}{
   |t-s|^{\tau}|x-y|^{\lambda}}\\
 &\quad+ \sup_{\substack{a\le s<t\le b\\ x\in K}}\frac{
\left|W(s,x)-W(t,x)\right|}{
  |t-s|^{\tau}}+\sup_{\substack{a\le t\le b\\ x,y\in K;x\neq y}}\frac{
\left|W(t,y) - W(t,x)\right|}{
  |x-y|^{\lambda}} \,,
\end{split}
\end{equation*}
is finite.
\end{enumerate}
However, for simplicity, we only employ \ref{cond.w}.



One of our main results in  this section is to define
$\int_{a}^{b}W(\dt,\varphi_{t})$ under the condition
$\la \ga+\tau>1$ through fractional integration by parts technique.
The following definition is  motivated from   Lemma \ref{lemma.fdg}.
\begin{definition}\label{defn.fracint}
    We define
    \begin{equation}\label{eqn.def.fintw}
        \int_{a}^{b}W(\dt,\varphi_{t})=(-1)^{\alpha}
        \int_{a}^{b}D_{a+}^{\alpha,t'}D_{b-}^{1-\alpha,t}W_{b-}(t,\varphi_{t'})|_{t'=t}\dt
    \end{equation}
    whenever the right hand side makes sense.
\end{definition}
\begin{remark}\label{r.3.relation}
Assume $d=1$.   Let
$W(t, x)=g(t) x$ be of the product form   and let $\varphi(t)=f(t)$,
where $g$ is a H\"older continuous function of exponent
$\tau$ and $f $ is a H\"older continuous function of exponent  $\la$.
If $1-\tau<\al<\la$,  then
\begin{align*}
\int_a^b W(\dt, \varphi_t)
&=(-1)^{\alpha}\int_{a}^{b}D_{a+}^{\alpha,t'}D_{b-}^{1-\alpha,t}W_{b-}(t, {t'})|_{t'=t}\dt\\
&= (-1)^{\alpha}\int_{a}^{b} D_{b-}^{1-\alpha,t}g_{b-} (t) D_{a+}^{\alpha, t}f( t)\dt\,.
\end{align*}
Thus from \eref{eq.young},
$\int_a^b W(\dt, \varphi_t)$ is  an extension of the classical
Young's integral $\int_a^b f(t) \d g(t)$ (see \cite{hustochastics},
\cite{young}, \cite{zahle}).   For general $d$,  if $W(t,x)=\sum_{i=1}^d g_i(t)x_i$
and $\varphi_i(t)=f_i(t)$,  then it is easy to see that
$\int_a^b W(\dt, \varphi_t)=\displaystyle
\sum_{i=1}^d \int_a^b f_i(t) \d g_i(t)$.
%
\end{remark}
The following result clarifies the context in which Definition \ref{defn.fracint} is justified.
\begin{theorem}\label{fractional}
 Assume the conditions  \ref{cond.w} and \ref{cond.phi} are satisfied. In addition, we suppose that
$\la \ga+\tau>1$. Let $\alpha\in(1- \tau,\lambda \tau)$. Then the
 right hand side of \eqref{eqn.def.fintw} is finite and is
 independent of $\alpha\in(1-\tau\,,  \la)$. As a consequence, we have
\begin{equation}\label{eq:def.w}
\begin{split}
&\int_{a}^{b}W(\dt,\varphi_{t})\\
&=(-1)^{\alpha}\int_{a}^{b}D_{a+}^{\alpha,t'}D_{b-}^{1-\alpha,t}W_{b-}(t,\varphi_{t'})|_{t'=t}\dt
\\
& = -\frac{1}{\Gamma(\alpha)\Gamma(1-\alpha)}\left\{ \int_{a}^{b}\frac{W_{b-}(t,\varphi_{t})}{(b-t)^{1-\alpha}(t-a)^{\alpha}}\dt\right.\\
&\left.
+\alpha\int_{a}^{b}\int_{a}^{t}\frac{W_{b-}(t,\varphi_{t})-W_{b-}(t,\varphi_{r})}
{(b-t)^{1-\alpha}(t-r)^{\alpha+1}}\dr\dt\right. \\
   &+ (1-\alpha)\int_{a}^{b}\int_{t}^{b}\frac{W(t,\varphi_{t})-W(s,\varphi_{t})}
 {(s-t)^{2-\alpha}(t-a)^{\alpha}}\ds\dt \\
   & + \left.\alpha(1-\alpha)\int_{a}^{b}\int_{a}^{t}\int_{t}^{b}
 \frac{W(t,\varphi_{t})-W(s,\varphi_{t})-W(t,\varphi_{r})+W(s,\varphi_{r})}{(s-t)^{2-\alpha}
 (t-r)^{\alpha+1}}\ds\dr\dt\right\},
\end{split}
\end{equation}
where $W_{b-}\left(t,x\right)=W\left(t,x\right)-W\left(b,x\right)$.
Moreover, there is a    universal constant $\kappa$ depending only on $\tau, \la$ and
 $\al$, but independent $W$,   $\varphi$ and $a$,  $b$ such that
 \begin{align}
\left| \int_{a}^{b}W(\dt,\varphi_{t})\right|
\le   \kappa  \|W\|_{\tau, \la\,;  a, b} (b-a)^{\tau }
 +\kappa  \|W\|_{\tau, \la\,;  a, b}
\|\varphi\|_{\ga\,;a, b}^\la (b-a)^{\tau+\la \ga}\,.\label{holder}
\end{align}
\end{theorem}

\begin{proof}
We denote $\|W\|=\|W\|_{\tau,\lambda;a,b}$. First by the definitions
of fractional derivatives \eref{e.2.3} and \eref{e.2.4}, we have
\[
D_{b-}^{1-\alpha,t}W_{b-}(t,\varphi_{t'})=\frac{(-1)^{1-\alpha}}{\Gamma(\alpha)}
\left(\frac{W_{b-}(t,\varphi_{t'})}{(b-t)^{1-\alpha}}+(1-\alpha)\int_{t}^{b}\frac{W(t,
\varphi_{t'})-W(s,\varphi_{t'})}{(s-t)^{2-\alpha}}\ds\right)\,.
\]
and
\begin{align*}
&D_{a+}^{\alpha,t'}D_{b-}^{1-\alpha,t}W_{b-}(t,\varphi_{t'})\\
&=\frac{(-1)^{1-\al} }{\Gamma(\alpha)\Gamma(1-\alpha)}\left(\frac{1}{(t'-a)^{\alpha}}\frac{W_{b-}
(t,\varphi_{t'})}{(b-t)^{1-\alpha}}+\alpha\int_{a}^{t'}\frac{W_{b-}(t,\varphi_{t'})
-W_{b-}(t,\varphi_{r})}{(t'-r)^{\alpha+1}(b-t)^{1-\alpha}}\dr\right.\\
&\quad+\frac{1-\alpha}{(t'-a)^{\alpha}}\int_{t}^{b}\frac{W(t,\varphi_{t'})
-W(s,\varphi_{t'})}{(s-t)^{2-\alpha}}\ds\\
&\quad\left.\negmedspace+(1-\alpha)\int_{a}^{t'}\frac{\alpha}
{(t'-r)^{\alpha+1}}\int_{t}^{b}\frac{W(t,\varphi_{t'})-W(s,\varphi_{t'})-W(t,\varphi_{r})
+W(s,\varphi_{r})}{(s-t)^{2-\alpha}}\ds\dr\right).
\end{align*}
Thus the right hand side of \eqref{eqn.def.fintw} is
\begin{align}
&  -\frac{1}{\Gamma(\alpha)\Gamma(1-\alpha)}\left\{ \int_{a}^{b}\frac{W_{b-}(t,\varphi_{t})}{(b-t)^{1-\alpha}(t-a)^{\alpha}}\dt
+\alpha\int_{a}^{b}\int_{a}^{t}\frac{W_{b-}(t,\varphi_{t})-W_{b-}(t,\varphi_{r})}
{(b-t)^{1-\alpha}(t-r)^{\alpha+1}}\dr\dt\right.\nonumber \\
 & \quad + (1-\alpha)\int_{a}^{b}\int_{t}^{b}\frac{W(t,\varphi_{t})-W(s,\varphi_{t})}
 {(s-t)^{2-\alpha}(t-a)^{\alpha}}\ds\dt\nonumber \\
 &  \quad + \left.\alpha(1-\alpha)\int_{a}^{b}\int_{a}^{t}\int_{t}^{b}
 \frac{W(t,\varphi_{t})-W(s,\varphi_{t})-W(t,\varphi_{r})+W(s,\varphi_{r})}{(s-t)^{2-\alpha}
 (t-r)^{\alpha+1}}\ds\dr\dt\right\} \nonumber\\
 &=:I_1+I_2+I_3+I_4\,.\label{e.3.17}
\end{align}
The condition \ref{cond.w} implies
\begin{align}
I_1
&\le \kappa   \|W\|
 \int_a^b (b-t)^{\tau+\al-1}(t-a)^{-\al}\dt\nonumber\\
&=\kappa  \|W\|    (b-a)^\tau\,. \label{I1}
\end{align}
Similarly, we also have
\begin{align}
I_3
&\le \kappa   \|W\|    \int_a^b \int_t^b (s-t)^{\tau +\al-2} (t-a)^{-\al} \ds\dt\nonumber\\
&\le \kappa   \|W\|    (b-a)^{\tau } \,. \label{I3}
\end{align}
The assumptions \ref{cond.w} and \ref{cond.phi} also imply
\begin{align*}
 |W_{b-}(t, \varphi_t)-W_{b-}(t, \varphi_r)|&\le \kappa   \|W\|
   |b-t|^\tau  |\varphi_t-\varphi_r|^\la \\
& \le
\kappa   \|W\|    \|\varphi\|_\ga ^\la   |b-t|^\tau
|t-r|^{\la \ga}\,.
\end{align*}
This implies
\begin{align}
I_2
&\le  \kappa   \|W\|   \|\varphi\|_\ga ^\la  \int_a^b \int_a^t ((b-t)^{\tau+\al-1} (t-r)^{\la \ga-\al-1} \dr\dt\nonumber\\
&\le \kappa   \|W\|    \|\varphi\|_\ga ^\la     (b-a)^ {\tau+\la \ga}\,. \label{I2}
\end{align}
Using
\begin{align*}
|W(t,\varphi_{t})-W(s,\varphi_{t})-W(t,\varphi_{r})+W(s,\varphi_{r})|
\le \kappa   \|W\|    \|\varphi\|_\ga ^\la    |t-s|^{\tau}|t-r|^{\la \ga}\,,
\end{align*}
we can estimate
$I_4$   as follows.
\begin{align}
I_4
 &\le
\kappa   \|W\| \|\varphi\|_\ga ^\la    \int_{a}^{b}\int_{a}^{t}\int_{t}^{b}
 \frac{|t-s|^\tau  |t-r|^{\la \ga} } {(s-t)^{2-\alpha}
 (t-r)^{\alpha+1}} \ds\dr\dt \nonumber\\
 &\le \kappa   \|W\| \|\varphi\|_\ga ^\la     (b-a)^{\tau+ \la \ga  }\,. \label{I4}
\end{align}
The inequalities \eref{I1}-\eref{I4} imply  that for any $\al\in (1-\tau, \ga\la)$,
the right  hand side of \eref{eqn.def.fintw} is well-defined.  The inequalities \eref{I1}-\eref{I4}
also yield
\eref{holder}.

To show \eref{eq:def.w} is independent of $\al$ we suppose $\alpha',\alpha\in(1-\tau,\lambda\gamma)$, $\alpha'>\alpha$.
Denote   $\beta=\alpha'-\alpha.$ Using Lemma \ref{lem:ifst}, it
is straightforward to see that
\begin{align*}
&
(-1)^{\alpha}\int_{a}^{b}D_{a+}^{\alpha,t}D_{b-}^{1-\alpha,t'}W_{b-}(t,\varphi_{t'})|_{t'=t}\dt \\
& =  (-1)^{\alpha}\int_{a}^{b}I_{a+}^{\beta,t}D_{a+}^{\beta,t}D_{a+}^{\alpha,t}D_{b-}^{1-\alpha,t'}
W_{b-}(t,\varphi_{t'})|_{t'=t}\dt\\
 & = (-1)^{\alpha+\beta}\int_{a}^{b}I_{b-}^{\beta,t'}D_{a+}^{\alpha+\beta,t}D_{b-}^{1-\alpha,t'}
 W_{b-}(t,\varphi_{t'})|_{t'=t}\dt\\
 & = (-1)^{\alpha'}\int_{a}^{b}D_{a+}^{\alpha',t}I_{b-}^{\beta,t'}D_{b-}^{1-\alpha,t'}
 W_{b-}(t,\varphi_{t'})|_{t'=t}\dt\\
 & = (-1)^{\alpha'}\int_{a}^{b}D_{a+}^{\alpha',t}D_{b-}^{1-\alpha',t'}W_{b-}(t,\varphi_{t'})|_{t'=t}\dt\,.
\end{align*}
This proves the theorem.
\end{proof}

Now we can improve the equality \eref{holder} as in the following
theorem

\begin{theorem}\label{thm.intW}  Let the assumptions
\ref{cond.w} and \ref{cond.phi} be satisfied.  Let $a,b,c$ be real numbers such that $a\le c\le b$.
Then there is a constant
$\kappa$   depending only on $\tau, \la$ and
 $\al$, but independent $W$,   $\varphi$ and $a$,  $b$, $c$
 such that
\begin{equation}
\label{est.W.c}
\left| \int_{a}^{b}W(\dt,\varphi_{t})-W(b,\varphi_c)+W(a,\varphi_c) \right|
\le\kappa  \|W\|_{\tau, \la\,;  a, b}
 \|\varphi\|_{\ga\,;a, b}^\la (b-a)^{\tau+\la \ga}\,.
\end{equation}
\end{theorem}


\begin{proof}
Let $a\le c<d\le b$ and let $\tilde \varphi(t)=\varphi(c) \chi_{[c, d)}(t)$, where $\chi_{[c, d)}$ is the indicate  function
on $[c, d)$. Then
\[
W(t, \tilde \varphi(t'))=\begin{cases}
W(t,   \varphi(c) )&c\le t'<d\\ \\
W(t,0)& \hbox{elsewhere}  \,.
\end{cases}
\]
This means $W(t, \tilde \varphi(t'))=W(t, \varphi(c) ) \chi_{[c, d)}(t')$.
Hence, from \eref{eq.young}  we have
\begin{align*}
\int_a^b W(\dt, \tilde \varphi(t))
&=(-1)^{\alpha}\int_{a}^{b} D_{b-}^{1-\alpha,t}W_{b-}(t, \varphi(c)) D_{a+}^{\alpha,t'} \chi_{[c, d)}(t')
|_{t'=t}\dt\\
&=(-1)^{\alpha}\int_{a}^{b} D_{b-}^{1-\alpha,t}W_{b-}(t, \varphi(c)) D_{a+}^{\alpha,t} \chi_{[c, d)}(t)
\dt\\
&= W(d, \varphi(c))-W(c, \varphi(c))\,.
\end{align*}

Let $c$ be any point in $ [a, b]$.
Denote $\tilde W(t,x)=W(t,x)-W(t, \varphi_c)$.  Then $\tilde W$ satisfies
 \ref{cond.w}.  As in the equation \eref{e.3.17}, we have
\begin{align*}
\int_a^b W(\dt, \varphi_t) -W(b,\varphi_c)+W(a,\varphi_ c)
&=\int_a^b\tilde W(\dt, \varphi_t)\\
&=\tilde I_1+\tilde I_2+\tilde I_3+\tilde I_4\,.
\end{align*}
where $\tilde I_2 =I_2$ and $\tilde I_4 =I_4$ are the same as $I_2$ and $I_4$
in the proof of Theorem \ref{fractional}.     But
\begin{align*}
\tilde I_1& =  -\frac{1}{\Gamma(\alpha)\Gamma(1-\alpha)} \int_{a}^{b}\frac{W(t,\varphi_{t})-W(b,\varphi_{t})-W(t,\varphi_{c})+W(b,\varphi_{c})}{(b-t)^{1-\alpha}(t-a)^{\alpha}}\dt
 \\
\tilde I_3 &=  -\frac{(1-\alpha) }{\Gamma(\alpha)\Gamma(1-\alpha)} \int_{a}^{b}\int_{t}^{b}\frac{W(t,\varphi_{t})-W(s,\varphi_{t})-W(t,\varphi_{c})+W(s,\varphi_{c})}
 {(s-t)^{2-\alpha}(t-a)^{\alpha}}\ds\dt\,.
\end{align*}
From the assumptions  \ref{cond.w} and \ref{cond.phi} we see that
\begin{eqnarray*}
&&|W(t,\varphi_{t})-W(b,\varphi_{t})-W(t,\varphi_{c})+W(b,\varphi_{c})|\\
&&\qquad \le \kappa \|W\|_{\tau, \la\,;  a, b}
  \|\varphi\|_{\ga\,;a, b}^\la |b-t|^{\tau}|t-c|^{\la \ga}\\
&&\qquad \le  \kappa \|W\|_{\tau, \la\,;  a, b}
\|\varphi\|_{\ga\,;a, b}^\la |b-t|^{\tau}|t-a|^{\la \ga}\,.
\end{eqnarray*}
This implies that
\begin{equation}
\tilde I_1\le \kappa  \|W\|_{\tau, \la\,;  a, b} \|\varphi\|_{\ga\,;a, b}^\la (b-a)^{\tau+\la \ga}\,.
\label{tildei1}
\end{equation}
Similarly, we have
\begin{equation}
\tilde I_3\le \kappa  \|W\|_{\tau, \la\,;  a, b}\|\varphi\|_{\ga\,;a, b}^\la (b-a)^{\tau+\la \ga}\,.\label{tildei3}
\end{equation}
Combining these two inequalities
\eref{tildei1} and \eref{tildei3} with the inequalities \eref{I2} and \eref{I4} we have
\[
\left|\int_a^b \tilde W(\dt, \varphi_t)\right|\le
\kappa  \|W\|_{\tau, \la\,;  a, b}  \|\varphi\|_{\ga\,;a, b}^\la (b-a)^{\tau+\la \ga}\,,
\]
which yields \eqref{est.W.c}.
\end{proof}
\begin{theorem}\label{thm.intWl}  Let the assumptions
\ref{cond.w}   be satisfied.  Let $\varphi: [a, b]\rightarrow \RR^d$
satisfy
\begin{equation}
\left|\varphi(s)-\varphi(a)\right|\le  L   |s-a|^\ell \quad \forall s\in [a, b]\
\qquad{\rm and}\qquad \sup_{a\le t<s\le b}\frac{|\varphi(s)-\varphi(t)|}{
(s-t)^\gamma}\le L
\label{e.biggamma}
\end{equation}
for some  $\ell  \in (\gamma, \infty)$ and for some constant $L\in(0, \infty)$.
If $\tau+\la \gamma > 1$,  then for any $\be < 1+\frac{\la\ga+\tau-1}{\ga}\ell $
we have
\begin{equation}
\label{est.W.c-add}
\left| \int_{a}^{b}W(dt,\varphi_{t}) -W(b,\varphi_a)+W(a,\varphi_a)  \right| \le C  (b-a)^{\be }\,,
\end{equation}
here the constant $C$ does not depend on $b-a$.
\end{theorem}
\begin{proof} As in the proof of Theorem \ref{thm.intW} we express
$\int_{a}^{b}W(\dt,\varphi_{t})-W(b,\varphi_a)+W(a,\varphi_a) $ as the sum of the terms $\tilde I_j$, $j=1,2,3,4$ (we follow the notation there).
First, we explain how to proceed
with $\tilde I_4$.
We shall use $C$ to denote a generic constant independent of $b-a$. Denote
\[
J:=|W(t,\varphi_{t})-W(s,\varphi_{t})-W(t,\varphi_{r})+W(s,\varphi_{r})|
\]
First,  we know that   we have
\begin{equation}
J \le C     |t-s|^{\tau}|t-r|^{\la \gamma  }\,.\label{e.j1}
\end{equation}
On the other hand, we also have
\begin{eqnarray}
J
&\le& |W(t,\varphi_{t})-W(s,\varphi_{t})-W(t,\varphi_{a})+W(s,\varphi_{a})|
\nonumber \\
&&\qquad +
 |W(t,\varphi_{r})-W(s,\varphi_{r})-W(t,\varphi_{a})+W(s,\varphi_{a})|
 \nonumber \\
 &\le& C |t-s|^{\tau}\left[ |t-a|^{\la \ell} +|r-a|^{\la \ell} \right]\nonumber \\
 &\le&  C |t-s|^{\tau}  |t-a|^{\la \ell} \,\label{e.j2}
\end{eqnarray}
when $a\le r<t<s\le b$.  Therefore, from \eref{e.j1} and \eref{e.j2}
it follows that  for any $\be_1\ge 0$ and $\be_2\ge 0$ with
$\be_1+\be_2=1$, we have
\[
J\le C  |t-s|^{\tau}|t-r|^{\be _1\la \ga}  |t-a|^{\be_2 \la \ell }
\]
If we choose $\al$ and $\be_1$ such that
\begin{equation}
\tau+\al>1\,,\qquad \be_1\la \gamma-\al>0 
\label{e.al-tau-beta}
\end{equation}
then
\[
\tilde I_4\le C (b-a)^{\be_1\la \gamma+\be_2\la \ell +\tau}\,.
\]
For any $\be < 1+\frac{\la\ga+\tau-1}{\ga}\ell $
we can choose $\al$,  $\be_1$, and $\be_2$ such that
 \eref{e.al-tau-beta} is satisfied  and
 \[
\tilde I_4\le C (b-a)^{\be}\,.
\]
The   term $\tilde I_2$   can be handled in a similar   but easier way and similar bound can be obtained.

Now, let us consider  $\tilde I_3$.  We have
\begin{equation*}
|W(t, \varphi_t)-W(s, \varphi_t)-W(t, \varphi_a)+W(s, \varphi_a)| \le C |t-s|^\tau |t-a|^{\la \ell}\,.
\end{equation*}
This   easily  yields
\[
\tilde I_3\le C (b-a)^{\tau+\la \ell}\,.
\]
Similar estimate holds true for $\tilde I_1$.  However, it is easy to verify
$\tau+\la \ell>1+\frac{\la\ga+\tau-1}{\ga}\ell$ if $\ell> \gamma$.
The theorem is proved.
\end{proof}

For every $s,t$ in $[a,b]$, we put $\mu(s,t)=W(t,\varphi_s)-W(s,\varphi_s)$.
 Let $\pi=\{a=t_0<t_1<\cdots<t_n=b\}$ be a partition
of $[a,b]$ with mesh size $|\pi|=\max_{1\le i\le n }|t_{i}-t_{i-1}|$, one can consider the limit of the Riemann sums
\begin{equation*}
    \lim_{|\pi|\downarrow0}\sum_{i=1}^n \mu(t_{i-1},t_{i})
\end{equation*}
whenever it exists. A sufficient condition for convergence of the Riemann sums is provided by following two results of \cite{FeyelPradelle06}.
\begin{lemma}[The sewing map]\label{sew.lem}
  Let $\mu$ be a continuous function on $[0,T]^2$ with values in a Banach space $B$  and $\ep>0$. Suppose that $\mu$ satisfies
  \[
 |\mu(a,b)-\mu(a,c)-\mu(c,b)|\le K|b-a|^{1+\ep}\quad \forall \ 0\le a\le c\le b\le T \,.
  \]
  Then there exists a function $\JJ\mu(t)$ unique up to an additive constant such that
  \begin{equation}\label{est.sew}
  |\JJ\mu(b)-\JJ\mu(a)-\mu(a,b)|\le K (1-2^{-\ep})^{-1} |b-a|^{1+\ep}\quad \forall \  0\le a\le b\le T\,.
  \end{equation}
\end{lemma}
We adopt the notation $\JJ_a^b \mu=\JJ \mu(b)-\JJ \mu(a)$
\begin{lemma}[Abstract Riemann sum]
  Let $\pi=\{a=t_0<t_1<\cdots<t_m=b\}$ be an arbitrary partition of $[a,b]$ with   $|\pi|=\sup_{i=0,\dots,m-1}|t_{i+1}-t_{i}|$.  Define the Riemann sum \[J_\pi=\sum_{i=0}^{m-1}\mu(t_i,t_{i+1})\] then $J_\pi$ converges to $\JJ_a^b\mu$ as $|\pi|\downarrow 0$ .
\end{lemma}

Because $\tau+\lambda \gamma$ is strictly greater than 1, the estimate \eqref{est.W.c} together with the previous two Lemmas implies
\begin{proposition}
    As the mesh size $|\pi|$ shrinks to 0, the Riemann sums
    \begin{equation*}
        \sum_{i=1}^n \left[W(t_{i},\varphi_{t_{i-1}})-W(t_{i-1},\varphi_{t_{i-1}})\right]
    \end{equation*}
    converges to $\int_a^bW(dt,\varphi_t)$.
\end{proposition}
It is easy to see from here that
\[
\int_a^b W(\d  t, \varphi_t)=\int_a^c W(\d  t, \varphi_t) +\int_c^b W(\d
t, \varphi_t)\qquad \forall \
a<c<b\,.
\]
This together with \eref{holder} imply easily the following.
\begin{proposition}\label{prop.frac} Assume that   \ref{cond.w} and \ref{cond.phi}
hold with $\lambda\gamma+\tau>1$.  As a function of $t$,  the indefinite
integral $\left\{ \int_{a}^t W(\d s,\varphi_{s})\,, \ t\le a\le b\right\}$ is H\"older continuous
of exponent $\tau$.
\end{proposition}

Further properties can be developed. For instance, we study the dependence of the nonlinear Young integration $\int W(\ds,\varphi_s)$
with respect to the medium $W$ and the integrand $\varphi$. We state the following two propositions whose
proofs left for readers (see e.g. \cite{hule2015}).
\begin{proposition}\label{prop.int.w12}
  Let $W_1$ and $W_2$ be functions on $\RR\times\RR^d$ satisfying the condition \ref{cond.w}. Let $\varphi$ be a function in $C^\gamma(\RR;\RR^d)$ and assume that $\tau+\lambda \gamma>1$. Then
  \begin{multline*}
    |\int_a^b W_1(\ds,\varphi_s)-\int_a^b W_2(\ds,\varphi_s)|\le |W_1(b,\varphi_a)- W_1(a,\varphi_a)-W_2(b,\varphi_a)+W_2(a,\varphi_a)|\\
    +c(\|\varphi\|_\infty) [W_1-W_2]_{\beta,\tau,\lambda}\|\varphi\|_{\gamma}|b-a|^{\tau+\lambda \gamma}\,.
  \end{multline*}
\end{proposition}%
\begin{proposition}\label{prop.int.phi12}
  Let $W$ be a function on $\RR\times\RR^d$ satisfying the condition \ref{cond.w}. Let $\varphi^1$ and $\varphi^2$ be two functions in $C^\gamma(\RR;\RR^d)$ and assume that $\tau+\lambda \gamma>1$. Let $\theta\in(0,1)$ such that $\tau+\theta \lambda \gamma>1$. Then for any $u<v$
  \begin{multline*}
    |\int_u^v W(\ds,\varphi_s^1)-\int_u^v W(\ds,\varphi_s^2)|\\
    \le C_1[W]_{\tau,\lambda}\|\varphi^1- \varphi^2\|_{\infty}^\lambda|v-u|^\tau \\
    +C_2[W]_{\tau,\lambda}\|\varphi^1- \varphi^2\|_{\infty}^{\lambda(1-\theta)} |v-u|^{\tau+\theta\lambda \gamma}\,,
  \end{multline*}
  where $C_1$ is an absolute constant and $C_2=2^{1-\theta} C_1(\|\varphi^1\|_{\gamma}^\lambda+\|\varphi^1\|_{\gamma}^\lambda)^\theta$.
\end{proposition}

\section{Iterated nonlinear integral}
From  Remark \ref{r.3.relation} we see that
if $W(t,x)=\sum_{i=1}^d g_i(t)x_i$
and $\varphi_i(t)=f_i(t)$,  then
$\int_a^b W(\dt, \varphi_t)=\displaystyle
\sum_{i=1}^d \int_a^b f_i(t) \d g_i(t)$.
We know that  the multiple (iterated) integrals of the form
\[\int_{a\le s_1\le s_2\le \cdots\le s_n\le b}  \varphi (s_1, s_2, \cdots, s_n ) dg(s_1) dg(s_2)
\cdots dg(s_n)
\]
are well-defined and have applications in expanding the solutions of
 differential equations (see \cite{hustochastics}).   What is the extension of the above
 iterated integrals to the nonlinear integral?  To simplify the presentation, we consider the case $d=1$.
 General dimension can be considered in a similar way with more complex notations.

We   introduce the following  notation.  Let
\[
\De_{n,a,b}:=\left\{ (s_1, \cdots, s_n)\,; \ a\le s_1\le s_2\le\cdots\le s_n\le b\right\}
\]
be a simplex in $\RR^n$.

 \begin{definition}Let  $\varphi:\De_{n, a, b}\rightarrow \RR$ be a
continuous function.  For a fixed $s_n\in [a, b]$, we can consider
$\varphi(\cdot, s_n)$ as a function of $n-1$ variables.  Assume  we can define
$\int_{\De_{n-1, a, s_n } } \varphi(s_1, \cdots, s_{n-1}, s_n)W(ds_1, \cdot)
\cdots W(ds_{n-1}, \cdot)$,  which is a function of $s_n$, denoted by $\phi_{n-1}(s_n)$,  then
we define
\begin{equation}
\int_{a\le s_1\le\cdots\le s_ n\le b } \varphi(s_1, \cdots,   s_n)W(ds_1, \cdot)
\cdots W(ds_n, \cdot)=\int_a^b W(ds_n, \varphi_{n-1}(s_n))\,.
\end{equation}
 \end{definition}
In the case $W(t,x)=f(t) x$,  such iterated integrals have been studied in \cite{hustochastics},  where an important case is  when  $\varphi(s_1, \cdots, s_n)
=\rho(s_1)$ for some function $\rho$ of one variable.
This means   that $\varphi(s_1, \cdots, s_n) $ depends only on
the first variable. This case appears in the remainder term when one expands
the solution of  a differential equation      and
 can be dealt with in the following way.

Let $F_1,F_2,\dots,F_n$ be jointly H\"older continuous functions on $[a,b]^2$.
More precisely, for each $i=1,\dots, n$, $F_i$ satisfies
\begin{eqnarray}
 && |F_i(s_1,t_1)-F_i(s_2,t_1)-F_i(s_1,t_2)+F_i(s_2,t_2)|  \label{cond.jointF}\\
  &&\qquad \le
  \|F_i\|_{\tau,\lambda ;a,b}|s_1-s_2|^\tau|t_1-t_2|^\lambda \,,
  \quad \hbox{for all $s_1,s_2,t_1,t_2$ in $[a,b]$.}\nonumber
\end{eqnarray}
We assume that $\tau+\lambda >1$.

Suppose that $F$ is a function satisfying \eqref{cond.jointF} with $\tau+\lambda>1$. The nonlinear integral $\int_a^b F(ds,s)$ can be defined analogously to Definition \ref{defn.fracint}.  Moreover, for a H\"older continuous function $\rho$ of order $\lambda$, we set $G(s,t)=\rho(t)F(s,t)$, it is easy to see that
\begin{eqnarray*}
 &&|G(s_1,t_1)-G(s_2,t_1)-G(s_1,t_2)+G(t_1,t_2)|\\
 &\le& |\rho(t_1)-\rho(t_2)||F(s_1,t_1)-F(s_2,t_1)|\\
 &&\quad+|\rho(t_2)||F(s_1,t_1)-F(s_2,t_1)-F(s_1,t_2)+F_i(t_1,t_2)|\\
 &\le&(\|\rho\|_{\tau}\|+\|\rho\|_\infty)\|F\|_{\tau,\lambda }|s_1-s_2|^\tau|t_1-t_2|^\lambda .
\end{eqnarray*}
Hence, the integration $\int \rho(s)F(ds,s)$ is well defined. In addition, it follows from Theorem \ref{thm.intW} that the map $t\mapsto \int_a^t \rho(s)F(ds,s)$ is H\"older continuous of order $\tau$.

We have then easily
\begin{proposition} Let $\rho$ be a H\"older continuous function of order $\lambda$. Under the condition \eref{cond.jointF} and $\tau>1/2$, the iterated integral
 \begin{equation}\label{iterateIn}
  I_{a,b}(F_1,\dots,F_n)=\int_{a\le s_1\le\cdots\le s_ n\le b }  \rho(s_1) F_1(ds_1,s_1)F_2(ds_2,s_2)\cdots F_n(ds_n,s_n)
\end{equation}
is well defined.
\end{proposition}

In the simplest case when  $\rho(s)=1$ and $F_i(s,t)=f(s)$ for all $i=1,\cdots,n$, the above integral becomes
\begin{equation*}
  \int_{a\le s_1\le\cdots\le s_ n\le b }  df(s_1)\cdots df(s_n)=\frac{(f(b)-f(a))^n}{n!}.
\end{equation*}
Therefore, one would expect that
\begin{equation}
  |I_{a,b}(F_1,\dots,F_n)|\le \kappa \frac{|b-a|^{\gamma_n}}{n!}. \label{e.expected}
\end{equation}
This estimate turns out to be true   for  \eref{iterateIn}.
\begin{theorem}\label{t.4.3}  Let $F_1,\dots,F_n$ satisfy \eref{cond.jointF} and $\rho$ be H\"older
continuous with  exponent $\lambda$. We assume that $\rho(a)=0$.
Denote $\displaystyle \be=\frac{\la  +\tau-1}{\lambda}$ and
$\ell_n=\displaystyle \frac{\be^{n-1}-1}{\be-1}+\be^{n-1} (\tau+\la)$.
Then,  for any $\ga_n<\ell_n$,    there is a constant $C_n $, independent of
$a$ and $b$ (but may depends on $\ga_n$)  such that
\begin{equation}
 |I_{a,b}(F_1,\dots,F_n)|\le C_n  |b-a|^{\ga_n}\,.
 \end{equation}
\end{theorem}
\begin{proof} Denote
\[
 I_{a,s}^{(k)} (F_1,\dots,F_k)=\int_{a\le s_1\le\cdots\le s_ k\le s }  \rho(s_1) F_1(ds_1,s_1)F_2(ds_2,s_2)\cdots F_k(ds_k,s_k)\,.
\]
Thus,  we see by definition that
\begin{equation}
I_{a,s}^{(k+1)} (F_1,\dots,F_{k+1} )=\int_a^s  F_{k+1} (dr,
I_{a,r}^{(k)} (F_1,\dots,F_k))\,. \label{e.k+1}
\end{equation}
We prove this theorem by induction on $n$.
When $n=1$,  the theorem follows straightforward from \eref{est.W.c}
with the choice $c=a$. Indeed, we have $|I^{(1)}_{a,t}|\le C|t-a|^{\lambda+\tau}$ and $|I^{(1)}_{a,t}-I^{(1)}_{a,s}|\le C|t-s|^{\tau}$.

The passage from $n$ to $n+1$ follows from the application of \eref{est.W.c-add}
to \eref{e.k+1}  and this concludes the proof of the theorem.
\end{proof}

\begin{remark} The estimate of Theorem \ref{t.4.3}  also holds true for the iterated nonlinear
Young integral $ I_{a,b}^{(n)}(F_1,\dots,F_n) =\int_{a\le s_1\le\cdots\le s_ k\le s }
 F_1(ds_1, \rho(s_1))F_2(ds_2,s_2)\cdots F_n(ds_n,s_n) $,
where $I_{a,b}^{(k)}(F_1,\dots,F_k) =\int_a^b F_k(ds, I_{a,s}^{(k-1)}(F_1,\dots,F_{k-1}))$,
and $ I_{a,b}^{(1)}(F_1) =\int_a^b
 F_1(ds , \rho(s ))$.
\end{remark}

\end{document}